\documentclass[12pt]{article}

\usepackage{amssymb}
\usepackage{amsthm}
\usepackage{amsmath}
\usepackage{graphicx}
\usepackage{fullpage}
\usepackage{color}
\usepackage[T1]{fontenc}
 \numberwithin{equation}{section}

\theoremstyle{plain}
\newtheorem{thm}{Theorem}[section]
\newtheorem{cor}[thm]{Corollary}

\newtheorem{lem}[thm]{Lemma}
\newtheorem{prop}[thm]{Proposition}

\theoremstyle{definition}
\newtheorem{defn}[thm]{Definition}
\newtheorem{ex}[thm]{Example}

\theoremstyle{remark}

\newtheorem{rem}[thm]{Remark}

\newcommand{\N}{\mathbb{N}}
\newcommand{\R}{\mathbb{R}}

\newcommand{\id}{\text{id}_{\R^d}}
\newcommand{\cG}{{\cal G}}
\newcommand{\cU}{{\cal U}}
\newcommand{\cV}{{\cal V}}
\newcommand{\cMp}{{\cal M}_p}
\newcommand{\cPp}{{\cal P}_p(\R^d)}

\newcommand{\bp}{\begin{proof}[\ensuremath{\mathbf{Proof}}]}
\newcommand{\bs}{\begin{proof}[\ensuremath{\mathbf{Solution}}]}
\newcommand{\ep}{\end{proof}}

\begin{document}


\title{Value functions in the Wasserstein spaces:\\
finite time horizons}

\author{Ryan Hynd\footnote{Department of Mathematics, University of Pennsylvania, 209 South 33rd St. Philadelphia PA 19104, rhynd@math.upenn.edu.  Partially supported by NSF grants DMS-1004733 and DMS-1301628.}\;
and Hwa Kil Kim\footnote{Department of Mathematical Sciences,
Seoul National University.  Partially supported by BK21 PLUS SNU Mathematical Sciences Division.}}


\maketitle

\begin{abstract}
We study analogs of value functions arising in classical mechanics in the space of probability measures endowed with the Wasserstein metric $W_p$, for $1<p<\infty$.  Our main result is that each of these generalized value functions is a type of viscosity solution of an appropriate Hamilton-Jacobi equation, completing a program initiated by Gangbo, Nguyen, and Tudorascu. Of particular interest is a formula we derive for a generalized value function when the associated potential energy is of the form ${\cal V}(\mu)=\int_{\R^d}V(x)d\mu(x)$. This formula allows us to make rigorous a well known heuristic connection between Euler-Poisson equations and classical Hamilton-Jacobi equations. Further results are presented which suggest there is a rich theory to be developed of deterministic control in the Wasserstein spaces.
\end{abstract}


\section{Introduction}
In this paper, we study generalizations of value functions of the form
\begin{equation}\label{ClassAction}
u(x,t)=\inf\left\{g(\gamma(0)) +\int^t_0\left(\frac{1}{p}|\dot{\gamma}(s)|^p -V(\gamma(s))\right)ds: \gamma\in AC_p\left([0,t]; \R^d\right), \; \gamma(t)=x\right\}
\end{equation}
in the space of probability measures.  In formula \eqref{ClassAction}, $(x,t)\in \R^d\times[0,\infty)$, $p\in (1,\infty)$, $g, V\in C(\R^d)$, and
$AC_p\left([0,t]; \R^d\right)$ consists of absolutely continuous paths $\gamma: [0,t]\rightarrow\R^d$ such that $\dot\gamma\in L^p([0,t];\R^d)$. Recall that
when $p=2$, $u(x,t)$ is the value of an action integral, evaluated along an optimal path, that arises in classical mechanics; in this case, the function $V$ has
a natural interpretation as potential energy.   Consequently, any function
$u$ as defined above will be called a {\it classical value function}.

\par An important fact about classical value functions is that they can be characterized as viscosity solutions of the classical Hamilton-Jacobi equation (HJE)
\begin{equation}\label{classicalHJE}
\partial_t u +\frac{1}{q}|\nabla u|^q +V(x)=0, \quad (x,t)\in\R^d\times(0,T),
\end{equation}
subject to the initial condition
$$
u(x,0)=g(x), \quad x\in \R^d.
$$
Standard references for this topic include \cite{Bardi, CIL, FS}. Here and throughout this paper, $q\in (1,\infty)$ is the conjugate H\"{o}lder exponent to $p$, $1/p+1/q=1$.

\par Another application of classical value functions is in designing
action minimizing trajectories. A necessary condition on any minimizing path $\gamma$ for $u(x,t)$ is that it satisfies the Euler-Lagrange equations
\begin{equation}\label{ClassEulLag}
\frac{d}{ds}\left(|\dot\gamma(s)|^{p-2}\dot\gamma(s)\right)=-\nabla V(\gamma(s)), \quad s\in (0,t).
\end{equation}
It turns out that optimality also necessitates
\begin{equation}\label{ClassOptEq}
|\dot\gamma(s)|^{p-2}\dot\gamma(s)=\nabla u(\gamma(s),s), \quad s\in (0,t)
\end{equation}
provided that $u$ is differentiable at each $(\gamma(s),s)$.

\par In this work, we will establish analogs of these properties for {\it generalized value functions}, which are functionals of the form
\begin{equation}\label{GenAction}
{\cal U}(\mu, t)=\inf\left\{{\cal G}(\sigma(0))+\int^t_0\left(\frac{1}{p}||\dot{\sigma}(s)||^p -{\cal V}(\sigma(s))\right)ds: \sigma\in AC_p([0,t], \cMp),\; \sigma(t)=\mu\right\}.
\end{equation}
Here $t\ge 0$ and $\cMp$ is the $p$-Wasserstein space. That is, $\cMp$ is the space ${\cal P}_p(\R^d)$ of Borel probability measures $\mu$ on $\R^d$ with finite $p$-th moments
$$
\int_{\R^d}|x|^pd\mu(x)<\infty
$$
endowed with the $p$-Wasserstein metric
$$
W_p(\mu,\nu):=\inf\left\{ \left(\iint_{\R^d\times\R^d}|x-y|^pd\pi(x,y)\right)^{1/p}: \pi\in \Gamma(\mu,\nu)\right\}.
$$
As usual, $\Gamma(\mu,\nu)$ is the subcollection of probability measures of ${\cal P}_p(\R^d\times\R^d)$ having first marginal $\mu$ and second marginal $\nu$.
We refer the reader to the volumes \cite{AGS,V} for more on the spaces $\cMp$.

\par In the expression \eqref{GenAction}, $\cG, \cV \in C(\cMp)$, and the space $AC_p([0,t], \cMp)$ consists of paths
$\sigma: [0,t]\rightarrow \cMp$ for which there is $h\in L^p([0,t])$ satisfying
$$
W_p(\sigma(s_1),\sigma(s_2))\le \int^{s_2}_{s_1}h(u)du
$$
for $0\le s_1\le s_2\le t$. The smallest such $h$ is denoted $||\dot\sigma||$ is called the {\it metric derivative} of $\sigma$.  Any Borel measurable map $v: \R^d\times [0,t]\rightarrow \R^d$
for which the continuity equation
\begin{equation}\label{ContEq}
\partial_s\sigma +\nabla\cdot (\sigma v)=0, \quad \R^d\times (0,t)
\end{equation}
holds in the sense of distributions,  and
\begin{equation}\label{VinLp}
\int^t_0 \int_{\R^d}|v(x,s)|^pd\sigma_s(x) ds<\infty,
\end{equation}
is known as a {\it velocity} for $\sigma$.  It turns out that there is always one velocity for $\sigma$ satisfying
\begin{equation}\label{MetricDer}
||v(s)||_{L^p(\sigma(s))}=||\dot\sigma(s)||
\end{equation}
for Lebesgue almost every $s\in [0,t]$ and we call it a {\it minimal velocity} (Theorem 8.3.1 of \cite{AGS}).

\par A standing assumption we will make throughout this paper is that the initial condition $\cG$ is Lipschitz continuous
\begin{equation}\label{GLip}
L:=\text{Lip}(\cG)<\infty.
\end{equation}
We will also assume throughout that there are $\alpha, \beta \in \R$, and $\varrho\in \cMp$ such that the potential $\cV$ satisfies the inequality
\begin{equation}\label{Vgrowth}
{\cal V}(\mu)\le \alpha W_p(\mu, \varrho)^p +\beta, \quad \mu\in \cMp.
\end{equation}
These are analogous to the assumptions typically made on $g$ and $V$ in \eqref{ClassAction} to ensure that a classical value function $u$ is finite valued. Under these conditions, our main result is as follows.
\begin{thm}\label{thm1} There is a positive number $T=T(\alpha,p)$ for which:\\
(i) $\cU\in C(\cMp\times (0,T))$; \\
(ii) $\lim_{t\rightarrow 0^+}\cU(\mu,t)={\cal G}(\mu)$; \\
(iii) $\cU$ is a viscosity solution (see definition \eqref{ViscDefn}) of the HJE
\begin{equation}\label{WassHJE}
\partial_t\cU +\frac{1}{q}||\nabla_\mu \cU||^q_{L^q(\mu)}+{\cal V}(\mu)=0
\end{equation}
on $\cMp \times (0, T)$.
\end{thm}
\par The motivation for this study originates in a paper of Gangbo, Nyugen, and Tudorasco \cite{G}.  These authors introduced the value function above and defined
a very natural notion of viscosity solution that is essentially the one we give in Definition \ref{ViscDefn}.
Their central result is that if $\cV$ satisfies \eqref{Vgrowth}, the value function is in general a viscosity subsolution and that if the dimension of the underlying Euclidean
space is one, then the value function is indeed a viscosity solution. They also derived the Euler-Poisson equations as a necessary condition for action minimizing trajectories.
The paper in question \cite{G} built on their previous work where they also considered questions related to fluid mechanics and action minimizing paths \cite{G2}. We improve the results of \cite{G} first by showing that generalized value functions are indeed viscosity solutions. We also show that for certain potentials $\cV$ and initial conditions $\cG$, there is a natural connection between the Euler-Poisson equations and the classical HJE equation. Moreover, in this case there is a simple, 
near explicit formula for the action. Finally, our results apply for any $p\in (1,\infty)$ while the work of \cite{G} only considered $p=2$.

\par There has been a recent plethora of research done on Hamilton-Jacobi equations in the space of measures and more generally in metric spaces; 
see for example \cite{FK}, \cite{FKu}, \cite{FN}, \cite{FS}, \cite{GS2}.   In particular, the notion of solution given in this paper is one of many possibilities.  For instance, the work by Giga-Hamamuki-Nakayasu \cite{GHN} is based on 
properties of solutions along curves, while the papers by Ambrosio and Feng \cite{AF} and Gangbo and \'{S}wi\k{e}ch \cite{GS} use the local slope of a function on a metric space to define a type of 
solution.  In \cite{AF}, the authors exploited the special geometry of ${\cal M}_2$ to show that seemingly different notions of weak solutions coincide which yields a comparison principle among viscosity sub and supersolutions.  Comparison also allowed Gangbo and \'{S}wi\k{e}ch \cite{GS, GS2} to develop an existence theory via the classical Perron method. Unfortunately, we do not have a convenient way of comparing viscosity sub and supersolutions as defined in section \ref{SupersolnProof}. This is an open problem, and we 
hope to resolve it in a forthcoming work.

\par We were also motivated by the potential application of this work in the theory of mean field games \cite{LL,LL2}. According to the
this theory, a version of the model -- or ``Master" -- equation is for a function  $\cU: (0,T]\times \R^d\times{\cal M}_2\rightarrow \R$ satisfying
\begin{equation}\label{MasterEq}
-\partial_t\cU+ \frac{1}{2}|\nabla_x\cU|^2 +\int_{\R^d}\nabla_x\cU\cdot \nabla_\mu \cU d\mu = {\cal F}(x,\mu), \quad (t,x,\mu)\in (0,T)\times\R^d\times{\mathcal M_2}
\end{equation}
and the terminal condition
$$
\cU(T,x,\mu)=\cG(x,\mu).
$$
While this equation is not in the exact form of the HJE studied in this paper, our hope is that the methods we developed will shed some light on the difficult equation \eqref{MasterEq}.

\par Another fundamental question that inspired this study was: do minimizing paths exist for $\cU(\mu,t)$?
Using a compactness argument, we establish existence under the additional assumption that $\cG$ and $\cV$ are continuous with respect to the narrow
topology on $\cPp$; see proposition Corollary \ref{Existence}. We also sought to establish some versions of this existence result even if $\cG$ and $\cV$
are not necessarily narrowly continuous.
\par A simple setting for this problem is when
\begin{equation}\label{narrowGV}
{\cal G}(\mu)=\int_{\R^d}g(x)d\mu(x)\quad\quad \text{and}\quad\quad {\cal V}(\mu)=\int_{\R^d}V(x)d\mu(x).
\end{equation}
For instance, recall that if $V(x)=|x|^p$, then ${\cal V}$ defined above is continuous on $\cMp$ but it is not narrowly continuous (see remark 7.1.11 of \cite{AGS}). Nevertheless, we are able to make
an interesting statement in this direction. In particular, we are able to provide a formula \eqref{ActForm} for the corresponding action $\cU$ in terms of the classical action $u$.
\begin{thm}\label{thm2} Assume $g$ is Lipschitz continuous on $\R^d$, $\inf V>-\infty$, $V(x)=O(|x|^p)$ as $|x|\rightarrow \infty$, and define $\cG$ and $\cV$ by \eqref{narrowGV}. Then there is
a $T=T(V,p)>0$ for which:
 \\
(i) the Wasserstein action integral $\cU$ is given by the formula
\begin{equation}\label{ActForm}
\cU(\mu,t)=\int_{\R^d}u(x,t)d\mu(x),\quad (\mu,t)\in \cMp\times [0,T)
\end{equation}
where $u$ is the classical action defined in equation \eqref{ClassAction};  \\
(ii) for each $t\in (0,T)$, there is a Borel map $\Psi: \R^d\times [0,t]\rightarrow \R^d$ such that for each
$x\in \R^d$, $s\mapsto \Psi(x,s)$ is a minimizer for $u(x,t)$ and
\begin{equation}\label{optSig}
 \sigma(s):=\Psi(s)_{\#}\mu, \quad  s\in [0,t]
\end{equation}
is a minimizing path for $\cU(\mu,t)$.
\end{thm}

\par When $\cG, \cV$ satisfy \eqref{narrowGV} for $g,V\in C^1(\R^d)$, we shall also see that any minimizing trajectory $\sigma$ satisfies the Euler-Possion equations which consist
of the continuity equation \eqref{ContEq} and the PDE
\begin{equation}\label{ClassEulerPoiss}
\partial_s(\sigma |v|^{p-2}v)+\nabla\cdot (\sigma |v|^{p-2}v\otimes v)=-\sigma \nabla V,\quad (x,s)\in \R^d\times (0,t).
\end{equation}
Observe this equation is a type of generalization of the classical Euler-Lagrange equations \eqref{ClassEulLag}.
Furthermore, it is not hard to see that if we formally differentiate the classical HJE with respect to the spatial variable, then $v=|\nabla u|^{q-2}\nabla u$ satisfies \eqref{ClassEulerPoiss}.
Fortunately, this heuristic observation can be made precise.

\begin{prop}\label{vequalDu} Assume the hypotheses of Theorem \ref{thm2} and that the classical value function $u$ is differentiable along its minimizing trajectories. Then any minimal velocity $v$ associated with
a minimizing path $\sigma$ for $\cU(\mu,t)$ satisfies
\begin{equation}\label{GradientStructure}
|v(x,s)|^{p-2}v(x,s) = \nabla u(x,s), \quad \sigma(s)\; \text{a.e. $x$}\in \R^d
\end{equation}
for Lebsegue almost every $s\in[0,t]$.
\end{prop}
\noindent
We remark that conditions can be imposed on $g$ and $V$ to ensure a classical value function $u$ is differentiable along its minimizing trajectories. 
See for instance section I.9 and I.10 of \cite{FS} or section 3.3 of \cite{Bardi}.  We also find it is interesting to compare the above proposition to the classical condition for optimality \eqref{ClassOptEq}.

\par The organization of this paper is as follows. In section \ref{PropSec}, we establish various properties of generalized actions including
dynamic programming, some simple estimates and continuity assertions. We go on to present some examples of generalized value functions and prove Theorem \ref{thm2} in section \ref{examples}; we also give some
computations to suggest that a more general theory exists than what is presented in this paper.  In section \ref{MinPaths}, we investigate properties
of minimizing paths and verify a useful compactness theorem. Finally, in section \ref{SupersolnProof} we verify Theorem \ref{thm1} and justify that $\cU$ is a viscosity solution of the HJE \eqref{WassHJE}.  The authors especially thank Wilfrid Gangbo and Andrzej \'{S}wi\k{e}ch for their insightful discussions, and the authors appreciate the hospitality of Georg-August-Universit\"{a}t G\"{o}ttingen and the Universitat Polit\`{e}cnica de Catalunya, which where most of this paper was written.

\section{Basic properties}\label{PropSec}
The classical action \eqref{ClassAction} is finite valued on some interval $[0,T)$ provided $V(x)=O(|x|^p)$, as $|x|\rightarrow \infty$.  This follows easily from
a version of the Poincar\'{e} inequality.  We note that it may be that $\lim_{t\uparrow T}u(x,t)=-\infty$; for instance, see Examples \ref{Power2Ex} and \ref{PowerpEx} below. Consequently, we
expect to require a similar condition on $\cV$ for the generalized action $\cU$ to be finite valued on an interval $[0,T)$; this is precisely why
we are assuming \eqref{Vgrowth}. \par In order to verify this claim, establish dynamic programming and verify continuity properties of generalized action, we will need a version of the Poincar\'{e} inequality in the space of measures. The proof shall be omitted
as it turns out to be straightforward and only a slight variation of Proposition 2 proved in \cite{G2}.
\begin{lem} For each $p\in [1,\infty)$, there is a positive constant $C_p$ such that
\begin{equation}\label{Poincare}
\left(\int^T_0W_p(\sigma(t),\sigma(T))^pdt\right)^{1/p}\le C_p T\left(\int^T_0||\dot\sigma(t)||^pdt\right)^{1/p}
\end{equation}
for each $\sigma\in AC_p((0,T), \cMp)$.
\end{lem}

\begin{cor}\label{UFinite}
Suppose $T$ satisfies
\begin{equation}\label{TChoice}
p(2C_pT)^p\alpha <1
\end{equation}
where $\alpha$ is the constant in \eqref{Vgrowth} and $C_p$ is the constant from the previous lemma. Then for each $(\mu,t)\in \cMp\times (0,T)$,
$$
\cU(\mu,t)>-\infty.
$$
\end{cor}
\begin{proof}
By our assumptions on ${\cal G}$ \eqref{GLip} and ${\cal V}$ \eqref{Vgrowth},
\begin{align*}
{\cal G}(\sigma(0))+\int^t_0\left(\frac{1}{p}||\dot{\sigma}(s)||^p -{\cal V}(\sigma(s))\right)ds &\ge {\cal G}(\mu) - LW_p(\sigma(0),\mu) \\
&\quad +\int^t_0\left(\frac{1}{p}||\dot{\sigma}(s)||^p -(\alpha(W_p(\sigma(s),\varrho)^p+\beta)\right)ds\\
&\ge {\cal G}(\mu) - LW_p(\sigma(0),\mu) -\left(\beta +2^pW_p(\varrho,\mu)^p\right) t\\
&\quad +\int^t_0\left(\frac{1}{p}||\dot{\sigma}(s)||^p -\alpha 2^pW_p(\sigma(s),\mu)^p\right)ds
\end{align*}
for any path $\sigma$ that is admissible for $\cU(\mu, t)$. The Poincar\'{e} inequality \eqref{Poincare} then implies
\begin{align}\label{PoincareLower}
{\cal G}(\sigma(0))+\int^t_0\left(\frac{1}{p}||\dot{\sigma}(s)||^p -{\cal V}(\sigma(s))\right)ds &\ge{\cal G}(\mu) - LW_p(\sigma(0),\mu) -\left(\beta +2^p\alpha W_p(\varrho,\mu)^p\right) t \nonumber \\
&\quad +\left(\frac{1}{p}-(2C_p T)^p\alpha\right)\int^t_0||\dot{\sigma}(s)||^pds \\
 &\ge{\cal G}(\mu) - LW_p(\sigma(0),\mu) -\left(\beta +2^p\alpha W_p(\varrho,\mu)^p\right) t \nonumber \\
&\quad +\left(\frac{1}{p}-(2C_p T)^p\alpha\right)\frac{W_p(\sigma(0),\mu)^p}{t^{p-1}}\nonumber \\
&\ge {\cal G}(\mu) -\left(\beta +2^p\alpha W_p(\varrho,\mu)^p\right)t \nonumber \\
& +t\inf_{z\ge 0}\left\{-Lz+ \left(\frac{1}{p}-(2C_p T)^p\alpha\right)z^p\right\}.\nonumber
\end{align}
We conclude as this lower bound is independent of $\sigma$.
\end{proof}
\begin{rem}
It is not difficult to see that if ${\cal V}$ satisfies an inequality such as \eqref{Vgrowth} for any power less than $p$, then we can use H\"{o}lder's inequality to make $\alpha$ as small as desired in \eqref{Vgrowth} for an appropriate $\beta$.
In this case, we can choose $T$ as large as desired and so $\cU(\mu,t)>-\infty$ for all $(\mu,t)\in \cMp\times [0,\infty)$.
\end{rem}\label{YoungRem}
The next fundamental property of the value function we present is the {\it dynamic programming principle}, which is a relationship between the functional $\mu\mapsto \cU(\mu,t)$ and its values at previous times $s\in [0,t]$. We omit the proof here as a standard argument from classical, deterministic control theory applies (Lemma 4.1 \cite{FS}); this was discovered in the paper \cite{G} for $p=2$.
\begin{prop}
Assume $T$ satisfies \eqref{TChoice}. For $0\le s\le t< T$, and $\mu\in M_p$
\begin{equation}\label{DPP}
\cU(\mu, t)=\inf\left\{\cU(\sigma(s),s)+\int^t_s\left(\frac{1}{p}||\dot{\sigma}(r)||^p -{\cal V}(\sigma(r))\right)dr: \sigma\in AC_p([s,t], M_p),\; \sigma(t)=\mu\right\}.
\end{equation}
\end{prop}
However, we will provide a careful proof of continuity of the action on $\cMp\times (0,T)$ below, as there is a slight oversight in the argument given in \cite{G}. Namely, a change of variables in Lemma 3.7 of \cite{G} was performed incorrectly.  However, the method presented in that lemma can be adapted without
too much difficulty.
\begin{prop}
$\cU$ is jointly continuous on $\cMp\times (0, T)$.
\end{prop}
\begin{proof}
Let $(\mu, t)\in \cMp\times (0,T)$ and $(\mu_n, t_n)$ be a sequence tending to $(\mu, t)$; without any loss of generality, let us suppose $t_n>0$ for each $n$. Assume $\sigma$ is an admissible path for $\cU(\mu,t)$ and for each $n>1$, define the path
$$
\sigma^n_\delta(\tau)=
\begin{cases}
\sigma\left(\frac{t}{(1-\delta)t_n}\tau\right), \quad 0\le \tau\le (1-\delta)t_n\\
\overline{\sigma}^n\left(\frac{\tau -(1-\delta)t_n}{\delta t_n}\right), \quad (1-\delta)t_n\le \tau\le t_n
\end{cases}
$$
where $\overline{\sigma}^n\in AC_p([0,1], \cMp)$ is a geodesic joining $\mu$ to $\mu_n$ and $\delta\in (0,1)$.  Clearly, $\sigma^n$ is
admissible for $\cU(\mu_n, t_n)$.

\par Consequently,
\begin{align*}
\cU(\mu_n,t_n) & \le {\cal G}(\sigma_\delta^n(0))+\int^{t_n}_0\left(\frac{1}{p}||\dot{\sigma}_\delta^n(\tau)||^p -{\cal V}(\sigma_\delta^n(\tau))\right)d\tau\\
&= {\cal G}(\sigma(0))+ \int^{(1-\delta)t_n}_0\left(\frac{1}{p}||\dot{\sigma}_\delta^n(\tau)||^p -{\cal V}(\sigma_\delta^n(\tau))\right)d\tau\\
&\quad + \int^{t_n}_{(1-\delta)t_n}\left(\frac{1}{p}||\dot{\sigma}_\delta^n(\tau)||^p -{\cal V}(\sigma_\delta^n(\tau))\right)d\tau \\
&= {\cal G}(\sigma(0)) +\left(\frac{t}{(1-\delta) t_n}\right)^{p-1}\int^{t}_0\frac{1}{p}||\dot{\sigma}(s)||^pds - \frac{(1-\delta)t_n}{t}\int^t_0{\cal V}(\sigma(s))ds\\
&+ \frac{1}{p(\delta t_n)^{p-1}}W_{p}(\mu,\mu_n)^p -\delta t_n\int^1_0 {\cal V}(\overline{\sigma}^n(s))ds
\end{align*}
Hence,
\begin{align*}
\limsup_{n\rightarrow \infty}\cU(\mu_n,t_n)&\le {\cal G}(\sigma(0)) +\left(\frac{1}{1-\delta}\right)^{p-1}\int^{t}_0\frac{1}{p}||\dot{\sigma}(s)||^pds-(1-\delta)\int^t_0{\cal V}(\sigma(s))ds.
\end{align*}
As $\delta\in(0,1)$ is arbitrary, we can send $\delta\rightarrow 0^+$ to conclude $\limsup_{n\rightarrow \infty}\cU(\mu_n,t_n)\le \cU(\mu,t)$.

\par Now choose a sequence of positive numbers $\epsilon_n$ tending to 0, as $n\rightarrow \infty$.  There is $\sigma_n$ admissible for $\cU(\mu_n, t_n)$ such that
$$
\cU(\mu_n,t_n)> - \epsilon_n +  {\cal G}(\sigma^n(0))+\int^{t_n}_0\left(\frac{1}{p}||\dot{\sigma}^n(\tau)||^p -{\cal V}(\sigma^n(\tau))\right)d\tau
$$
For $\delta\in (0,1)$, define the sequence
$$
\sigma^n_\delta(s)=
\begin{cases}
\sigma^n\left(\frac{t_n}{(1-\delta)t}s\right), \quad 0\le s\le (1-\delta)t\\
\overline{\sigma}^n\left(\frac{s -(1-\delta)t}{\delta t}\right), \quad (1-\delta)t\le s\le t
\end{cases}.
$$
Here $\overline{\sigma}^n\in AC_p([0,1], \cMp)$ is a geodesic connecting $\mu_n$ to $\mu$. Observe $\sigma^n_\delta$ is admissible for $\cU(\mu,t)$, and so
\begin{align*}
\cU(\mu,t) & \le {\cal G}(\sigma_\delta^n(0))+\int^{t}_0\left(\frac{1}{p}||\dot{\sigma}_\delta^n(s)||^p -{\cal V}(\sigma_\delta^n(s))\right)ds\\
&= {\cal G}(\sigma^n(0)) +\left(\frac{t_n}{(1-\delta) t}\right)^{p-1}\int^{t_n}_0\frac{1}{p}||\dot{\sigma^n}(\tau)||^pd\tau - \frac{(1-\delta)t}{t_n}\int^{t_n}_0{\cal V}(\sigma^n(\tau))d\tau\\
&+ \frac{1}{p(\delta t)^{p-1}}W_{p}(\mu_n,\mu)^p -\delta t\int^1_0 {\cal V}(\overline{\sigma}^n(s))ds\\
&<\epsilon_n +\cU(\mu_n,t_n) +  \frac{1}{p(\delta t)^{p-1}}W_{p}(\mu_n,\mu)^p -\delta t\int^1_0 {\cal V}(\overline{\sigma}^n(s))ds \\
& + \frac{1}{p}\left(\left(\frac{t_n}{(1-\delta) t}\right)^{p-1} -1\right)\int^{t_n}_0||\dot\sigma^n(\tau)||^pd\tau -\left(\frac{(1-\delta)t_n}{t}-1\right)\int^{t_n}_0{\cal V}(\sigma^n(\tau))d\tau.
\end{align*}
One checks that since $t_n<T$ for $n$ large, $\int^{t_n}_0||\dot\sigma^n(\tau)||^pd\tau$ is bounded independently of $n\in \N$. In particular, the
sequence of paths $\{\sigma^n(\tau)\in \cMp: 0\le \tau\le t_n\}$ is uniformly bounded. Hence, we may send $n\rightarrow \infty$ and $\delta\rightarrow 0^+$ above to conclude
$$
\cU(\mu,t) \le \liminf_{n\rightarrow \infty}\cU(\mu_n, t_n).
$$

\end{proof}
\begin{prop}
Assume ${\cal V}$ is uniformly continuous. Then $\mu\mapsto \cU(\mu,t)$ is uniformly continuous for each
$t\ge 0$.
\end{prop}

\begin{proof}
First, recall that $\cMp$ is a metric length space, and that any two points can be joined by a constant speed geodesic.
A direct result of these facts is that $\cV$ satisfies \eqref{Vgrowth} for $p=1$ and so $\mu\mapsto \cU(\mu,t)$ is continuous on $\cMp$ for each
$t\ge 0$. Next, assume $\omega_\cG$ and $\omega_\cV$ are moduli of continuity for ${\cal G}$ and ${\cal V}$, respectively.
Let us also initially suppose $\mu_1, \mu_2\in \cMp$ are absolutely continuous with respect to Lebesgue measure; it follows that there is a Borel
map $\Upsilon:\R^d\rightarrow \R^d$ for which $\Upsilon_\#\mu^2=\mu^1,$  $\Upsilon^{-1}_\#\mu^1=\mu^2$ and
$$W^p_p(\mu^1,\mu^2)= \int_{\R^d}|x-\Upsilon(x)|^pd\mu^2(x) = \int_{\R^d}|\Upsilon^{-1}(y)-y|^pd\mu^1(y).
$$
\par For a given $\epsilon >0,$ choose an admissible path $\sigma^2$ for $\cU(\mu_2,t)$ such that
$$
\cU(\mu^2,t)\geq - \epsilon +  {\cal G}(\sigma^2(0))+\int^{t}_0\left(\frac{1}{p}||\dot{\sigma}^2(\tau)||^p -{\cal V}(\sigma^2(\tau))\right)d\tau,
$$
and let $v:\mathbb{R}^d\times[0,t]\rightarrow \mathbb{R}^d$ be a velocity field for $\sigma^2$ satisfying
$$||v(s) ||_{L^p(\sigma^2(s))}=||\dot{\sigma}^2(s)|| .
$$
According to Theorem 8.2.1 of \cite{AGS}, there is a Borel probability measure $\eta$ on $\R^d\times \Gamma_t$ such that
$$
\sigma^2(s)=e(s)_{\#}\eta, \quad s\in [0,t].
$$
Here $\Gamma_t:=C([0,t]; \R^d)$ equipped with the supremum norm, and
\begin{equation}\label{EvalMaps}
e(s):\R^d\times \Gamma_t\rightarrow \R^d; (x,\gamma)\mapsto \gamma(s)
\end{equation}
(for $s\in [0,t]$).
Moreover, $\eta$ is concentrated on pairs $(x,\gamma)$ for which $\gamma$ is a solution of the ODE
\begin{equation}\label{gammadotv}
\dot{\gamma}(s)=v(\gamma(s),s), \quad \text{a.e.}\; s\in(0,t)
\end{equation}
satisfying $\gamma(t)=x$.

\par Define a map $\tilde{S} : \mathbb{R}^d\times \Gamma_t \rightarrow \mathbb{R}^d\times \Gamma_t$ by
$$\tilde{S}:(x,\gamma)\rightarrow (\Upsilon(x),\gamma +\Upsilon(x)-x),
$$
a Borel probability measure $\tilde{\mathbf{ \eta}}$ on $\mathbb{R}^d\times \Gamma_t$
$$
\tilde{\mathbf{ \eta}}:= \tilde{S}_{\#}\mathbf{ \eta},
$$
and a path
$$\sigma^1(s) = (e(s))_{\#}\tilde{\mathbf{ \eta}}, \quad s\in [0,t].
$$
Note $\sigma^1(0)=\mu^1$ and
\begin{align*}
 W^p_p(\sigma^1(s_1),\sigma^1(s_2))&\leq \int_{\R^d\times\Gamma_t}|\gamma(s_1)-\gamma(s_2)|^pd\tilde{\mathbf{ \eta}}(x,\gamma)\\
&=\int_{\R^d\times\Gamma_t}|\left (\gamma(s_1)+\Upsilon(x)-x\right )-\left(\gamma(s_2)+\Upsilon(x)-x\right)|^pd{\mathbf{ \eta}}(x,\gamma)\\
&\leq \int_{\R^d\times\Gamma_t}\left|\int_{s_1}^{s_2}\dot{\gamma}(s)ds\right|^pd{\mathbf{ \eta}}(x,\gamma)\\
&\leq (s_2-s_1)^{\frac{p}{q}}\int_{\R^d\times\Gamma_t}\int_{s_1}^{s_2}|\dot{\gamma}(s)|^pdsd{\mathbf{ \eta}}(x,\gamma)\\
&=(s_2-s_1)^{\frac{p}{q}}\int_{s_1}^{s_2}\int_{\R^d\times\Gamma_t}|v(\gamma(s),s)|^p d{\mathbf{ \eta}}(x,\gamma)ds\\
&=(s_2-s_1)^{\frac{p}{q}}\int_{s_1}^{s_2}||v(s)||^p_{L^p(\sigma^2(s))}ds.
\end{align*}
Thus, $\sigma^1$ is an admissible path for $\cU(\mu^1,t).$ Also notice
\begin{align*}
 \frac{W^p_p(\sigma^1(s_1),\sigma^1(s_2))}{|s_2-s_1|^p}\leq \frac {\int_{s_1}^{s_2}||v(s)||^p_{L^p(\sigma^2(s))}ds}{|s_2-s_1|}.
\end{align*}
As a result
$$||\dot\sigma^1(s)||\le ||\dot\sigma^2(s)||$$
for Lebesgue almost every $s\in [0,t]$.

\par Next
\begin{align*}
 W_p^p(\sigma^1(s),\sigma^2(s))&= W_p^p((e(s))_{\#}\tilde{\mathbf{ \eta}},(e(s))_{\#}{\mathbf{ \eta}})\\
&= W_p^p((e(s))_{\#}(\tilde{S}_{\#}{\mathbf{ \eta}}),(e(s))_{\#}{\mathbf{ \eta}})\\
&\leq\int_{\R^d\times\Gamma_t}|(\gamma(s)+\Upsilon(x)-x) -\gamma(s)   |^pd{\mathbf{ \eta}}(x,\gamma)\\
&=\int_{\R^d}|\Upsilon(x)-x|^pd \mu^2(x)\\
&=W_p^p(\mu^1,\mu^2).
\end{align*}
As a result,
\begin{align*}
\cU(\mu^1,t)-\cU(\mu^2,t) &\le \epsilon +  {\cal G}(\sigma^1(0))- {\cal G}(\sigma^2(0)) \\
& + \int^{t}_0\left(\frac{1}{p}||\dot{\sigma}^1(\tau)||^p - \frac{1}{p}||\dot{\sigma}^2(\tau)||^p\right)d\tau -\int^t_0\left({\cal V}(\sigma^1(\tau))-{\cal V}(\sigma^2(\tau))\right)d\tau\\
&\le \omega_\cG(W_p(\sigma^1(0),\sigma^2(0)) +\int^t_0\omega_\cV(W_p(\sigma^1(\tau),\sigma^2(\tau))d\tau +\epsilon\\
&\le \omega_\cG(W_p(\mu^1,\mu^2))+t\omega_\cV(W_p(\mu^1,\mu^2))+\epsilon\\
&=(\omega_\cG +t \omega_\cV)(W_p(\mu^1,\mu^2)) +\epsilon.
\end{align*}
\par Exchanging $\mu^1$ and $\mu^2$ and repeating the argument above gives
\begin{equation}\label{UnifContU}
|\cU(\mu^1,t)-\cU(\mu^2,t) |\le (\omega_\cG +t \omega_\cV)(W_p(\mu^1,\mu^2)) +\epsilon.
\end{equation}
The general assertion, when $\mu^1,\mu^2\in \cMp$ are not necessarily absolutely continuous, follows from the density of absolutely continuous measures in $\cMp$ (Lemma 7.1.10 in \cite{AGS}) and the continuity of $\cU$.
\end{proof}

\begin{cor}\label{LipU}
If ${\cal V}$ is Lipschitz continuous and bounded, then $\cU$ is jointly Lipschitz continuous on $\cMp\times [0,T]$ for each $T>0$.
\end{cor}
\begin{proof}
Fix $T>0$ and define $L_1:=\text{Lip}(\cG)+ T\text{Lip}(\cV)$. Inequality \eqref{UnifContU} above states
$$
|\cU(\mu^1,t)-\cU(\mu^2,t)|\le L_1 W_p(\mu^1,\mu^2)
$$
for $\mu^1,\mu^2\in \cMp$ and $t\in [0,T]$.  Using the dynamic programming identity \eqref{DPP}, we find for $0\le s\le t\le T$
\begin{align*}
\cU(\mu, t) & = \inf\left\{\cU(\sigma(s),s)+\int^t_s\left(\frac{1}{p}||\dot{\sigma}(r)||^p -{\cal V}(\sigma(r))\right)dr: \sigma\in AC_p([s,t], M_p),\; \sigma(t)=\mu\right\}\\
& \ge \cU(\mu, s) + \inf\left\{-L_1W_2(\sigma(s),\mu)+\int^t_s\left(\frac{1}{p}||\dot{\sigma}(r)||^p -\sup{\cal V}\right)dr:\; \sigma(t)=\mu\right\}\\
& \ge \cU(\mu, s) -(\sup{\cal V})(t-s)+\inf\left\{-L_1W_p(\sigma(s),\mu)+\frac{W_p(\sigma(s),\mu)^p}{p(t-s)^{p-1}}:\; \sigma(t)=\mu\right\}\\
& \ge \cU(\mu, s) -(\sup{\cal V})(t-s) + \inf_{z\ge0}\left\{-L_1z+\frac{z^p}{p}\right\}(t-s)\\
& \ge \cU(\mu, s) -\left[ \sup{\cal V}- \inf_{z\ge0}\left\{-L_1z+\frac{z^p}{p}\right\}\right](t-s).
\end{align*}
Dynamic programming also implies, by choosing $\sigma(r)=\mu, \; s\le r\le t$,
$$
\cU(\mu,t)\le \cU(\mu,s) - (t-s){\cal V}(\mu)\le \cU(\mu,s) + (-\inf{\cal V})(t-s).
$$
Therefore,
$$
|\cU(\mu,t)-\cU(\mu,s)|\le \max\left\{\left[ \sup{\cal V}-\inf_{z\ge0}\left\{-L_1z+\frac{z^p}{p}\right\}\right], -\inf{\cal V}\right\}(t-s).
$$
\end{proof}

\section{Examples}\label{examples}
In this section, we present some important examples and prove Theorem \ref{thm2} along the way.  In our estimate, these are the simplest examples of
generalized action functions that can be represented by formulae alternative to \eqref{GenAction}. The first example is a generalization of the classical Hopf-Lax type formula
which have also been studied in various metric spaces. The other types of examples involve what we believe is a new formula that is clearly specific to spaces
of measures. This formula will prove to be useful as it will help us to verify the existence of
minimizing paths for the corresponding action, when it is unclear if any sort of direct methods apply.

\begin{ex}
When ${\cal V}\equiv 0$, $\cU$ defined in \eqref{GenAction} is also given by
$$
\cU(\mu,t)=\inf_{\tau\in M_p}\left\{{\cal G}(\tau) + \frac{W_p(\mu,\tau)^p}{pt^{p-1}}\right\}, \quad (\mu,t)\in \cMp\times (0,\infty).
$$
This identity is well known and this type of Hopf-Lax formula holds in vast generality \cite{AF, Drag, Gos, V}. Moreover, the modified value function
\begin{equation}\label{ModValue}
\cU(\mu,t)=\inf\left\{\cG(\sigma(0)) + \int^t_0\ell(||\dot\sigma(s)||)ds: \sigma\in AC_p([0,t], {\cal M}_p),\; \sigma(t)=\mu\right\}
\end{equation}
also can be expressed as
\begin{equation}\label{GenHopfLax}
\cU(\mu,t)=\inf_{\tau\in {\cal M}_p}\left\{\cG(\tau) + t\ell\left(\frac{W_p(\mu,\tau)}{t}\right)\right\}, \quad (\mu,t)\in {\cal M}_p\times (0,\infty)
\end{equation}
provided $\ell:[0,\infty)\rightarrow [0,\infty)$ is increasing and convex.  This is a simple consequence of Jensen's inequality and the fact that
$\cMp$ is a length space.

\par A novelty of this work is that our methods can be used to establish, under mild assumptions, that $\cU$ \eqref{GenHopfLax} is a very natural type of solution of the PDE
\begin{equation}\label{LaxPDE}
\cU_t + \ell^*\left(||\nabla_\mu \cU||_{L^q(\mu)}\right)=0.
\end{equation}
Here $\ell^*$ is the Legendre transform of $\ell$; see Proposition \ref{LaxProp}. Along with the main result of this paper, we view this assertion as positive evidence
that there is a theory of deterministic control to be developed in the Wasserstein spaces.

\end{ex}
We now consider the case when $\cG$ and $\cV$ satisfy \eqref{narrowGV}.  We shall further assume for the remainder of this section that $g$ is Lipschitz and $V$ satisfies
\begin{equation}\label{classVassump}
|V(x)|\le a|x|^p+b, \quad x\in \R^d
\end{equation}
for some $a,b\in \R$.  It is easy to verify that $\cG$ satisfies \eqref{GLip} with $L\le\text{Lip}(g)$ and $\cV$ satisfies \eqref{Vgrowth} for any $\alpha>a$ and appropriate $\beta$ dependent on $a$ and $b$.  Thus, the results of the previous section holds with $T$ chosen to satisfy
\eqref{TChoice} with $\alpha=a$. We can also argue as we did in the previous section to conclude the classical action $u$ is continuous on $\R^d\times (0,T)$ for the same
choice of $T$.

\par As previously remarked, it is now well known that $u$ is a viscosity solution of the HJE \eqref{classicalHJE} on $\R^d\times (0,T)$. Another fact that we
shall make use of is that for each $(x,t)\in\R^d\times (0,T)$, $u(x,t)$ has a minimizing path.  This property of $u$ follows from our assumptions on $g$ and $V$ and well known compactness results \cite{FR, FS}.
We now aim  to verify Theorem \ref{thm2} which states the formula for the generalized value function turns out to be particularly simple if we assume \eqref{narrowGV}. In order to verify this result, we shall need a rather crucial lemma.

\begin{lem}
Fix $t\in (0,T)$, and for $x\in \R^d$ define the set valued map
$$
F(x):=\left\{\gamma\in AC_p\left([0,t]; \R^d\right): u(x,t)=g(\gamma(0)) +\int^t_0\left(\frac{1}{p}|\dot{\gamma}(s)|^p -V(\gamma(s))\right)ds,\; \gamma(t)=x\right\}.
$$
There is a Borel measurable mapping $\Phi: \R^d \rightarrow AC_p\left([0,t]; \R^d\right)$ such that $\Phi(x)\in F(x)$ for all $x\in \R^d$.
\end{lem}
\begin{proof}
Consider the Banach space $AC_p([0,t], \R^d)\subset C([0,t]; \R^d)$ with the norm
$$
||\gamma||=\max_{0\le s\le t}|\gamma(s)| + \left(\int^t_0|\dot\gamma(s)|^pds\right)^{1/p}.
$$
We employ Theorem 8.3.1 of \cite{AB}, which establishes that if $F$ has closed, nonempty images, $F$ has a Borel measurable selection provided the function
$$
x\mapsto \text{dist}(\eta,F(x))
$$
is Borel measurable for each $\eta\in AC_p([0,t], \R^d)$.  We show this function is in fact lower-semicontinuous.

\par Let $x\in \R^d$ and $\{x_n\}_{n\in\N}$ tending to $x$ as $n\rightarrow \infty$. Furthermore, choose a subsequence $\{x_{n_j}\}_{j\in \N}$ so that
$$
\liminf_{n\rightarrow \infty}\text{dist}(\eta,F(x_n))=\lim_{j\rightarrow \infty}\text{dist}(\eta,F(x_{n_j})).
$$
Let $\gamma_j\in F(x_{n_{j}})$ be such that $\text{dist}(\eta,F(x_{n_j}))=||\eta-\gamma_j||$; that such a $\gamma_j$ exists follows from a routine compactness argument 
(similar to the one given in this proof below). Also recall that
$$
u(x_{n_{j}},t)=g(\gamma_j(0))+\int^t_0\left(\frac{1}{p}|\dot{\gamma_j}(s)|^p -V(\gamma_j(s))\right)ds.
$$
In view of our assumptions on $g$ and $V$, and as $u(x_{n_{j}},t)\le g(x_{n_{j}})-t V(x_{n_{j}})$ is bounded from above, standard manipulations imply
$$
\sup_{j\in \N}\int^t_0|\dot{\gamma_j}(s)|^pds<\infty.
$$
\par It follows that $\gamma_j$ is  equicontinuous and uniformly bounded since
$$
|\gamma_j(s)|\le |\gamma_j(s)-\gamma_j(t)|+|\gamma_j(t)|\le\left(\int^t_s|\dot\gamma(\tau)|^pd\tau\right)^{1/p}|t-s|^{1-1/p}+|x_j|\le C|t-s|^{1-1/p}+|x_j|
$$
for all $j\in \N$ and $s\in [0,t]$. Consequently, there is a subsequence $\{\gamma_{j_k}\}$ that converges uniformly to
some $\gamma\in AC_p([0,t])$ and $\dot\gamma_{j_k}$ converges to $\dot\gamma$ weakly in $L^p([0,t])$, as $k\rightarrow \infty$. Moreover,
it is immediate from this convergence that $\gamma\in F(x)$.
\par Finally, notice that
\begin{align*}
\text{dist}(\eta,F(x))& \le ||\eta -\gamma|| \\
& = \max_{0\le s\le t}|\eta(s)-\gamma(s)| + \left(\int^t_0|\dot\eta(s)-\dot\gamma(s)|^pds\right)^{1/p}\\
& \le  \liminf_{k}\left(\max_{0\le s\le t}|\eta(s)-\gamma_{j_k}(s)| + \left(\int^t_0|\dot\eta(s)-\dot\gamma_{j_k}(s)|^pds\right)^{1/p}\right)\\
&= \liminf_{k}\text{dist}(\eta,F(x_{n_{j_k}}))\\
&= \lim_{j}\text{dist}(\eta,F(x_{n_{j}}))\\
&= \liminf_{n}\text{dist}(\eta,F(x_{n})).
\end{align*}
\end{proof}

Recall that if $\Phi(x)=\gamma$ at time $t>0$, and $u$ is differentiable at each $(\gamma(s),s)$, then $\gamma$ satisfies \eqref{ClassOptEq}
$$
|\dot\gamma(s)|^{p-2}\dot\gamma(s)=\nabla u(\gamma(s),s), \quad s\in (0,t).
$$
Again we denote $\Gamma_t=C([0,t]; \R^d)$ (equipped with the supremum norm) and notice the family of evaluation maps $e(s)$ defined in \eqref{EvalMaps} are continuous. As a result, the family of
composition mappings of $\R^d$
\begin{equation}\label{FlowMap}
\Psi(s): = e(s)\circ (\text{id}_{\R^d},\Phi)
\end{equation}
are also Borel measurable. We interpret this family of maps $\{\Psi(s)\}_{s\in[0,t]}$  as a flow of minimizing trajectories since for each $x\in \R^d$, $s\mapsto \Psi(x,s)$ is a solution of the above ODE.
We now use this flow map to furnish a proof of Theorem \eqref{thm2}.

\begin{proof} (of Theorem \eqref{thm2}) $(i)$ Let $\sigma$ be an admissible path for $\cU(\mu,t)$ and $v:\R^d\times[0,t]\rightarrow \R^d$ a velocity field for $\sigma$ satisfying \eqref{MetricDer}. By Theorem 8.2.1 of \cite{AGS}, there is a measure $\eta$ on $\R^d\times \Gamma_t$ such that $\sigma(s)=e(s)_{\#}\eta.$ Moreover, $\eta$ is concentrated on pairs $(x,\gamma)$ for which $\gamma$ is a solution of the ODE \eqref{gammadotv} satisfying $\gamma(t)=x$.  By Tonelli's theorem,
\begin{align*}
\int_{\R^d\times\Gamma_t}\left\{\int^t_0|\dot\gamma(s)|^pds\right\}d\eta(x,\gamma) &= \int^t_0\left\{\int_{\R^d\times\Gamma_t}|\dot\gamma(s)|^pd\eta(x,\gamma)\right\}ds \\
& = \int^t_0\left\{\int_{\R^d\times\Gamma_t}|v(\gamma(s),s)|^pd\eta(x,\gamma)\right\}ds\\
& = \int^t_0\left\{\int_{\R^d}|v(x,s)|^pd\sigma_s(x)\right\}ds\\
& = \int^t_0||\dot\sigma(s)||^pds<\infty.
\end{align*}
Consequently, $\eta$ is actually concentrated on a subset of $AC_p([0,t],\R^d)\subset \Gamma_t$.

\par Observe
\begin{align}\label{ImportantVequalDu}
{\cal G}(\sigma(0))+\int^t_0\left(\frac{1}{p}||\dot{\sigma}(s)||^p -{\cal V}(\sigma(s))\right)ds \nonumber
&= \int_{\R^d}g(x)d\sigma_0(x) +\int^t_0\int_{\R^d}\left(\frac{1}{p}|v(x,s)|^p -V(x)\right)d\sigma_s(x)ds \nonumber \\
&= \int_{\R^d\times \Gamma_t}g(\gamma(0))d\eta(x,\gamma)\nonumber \\
& \quad + \int^t_0\int_{\R^d\times\Gamma_t}\left(\frac{1}{p}|v(\gamma(s),s)|^p -V(\gamma(s))\right)d\eta(x,\gamma)ds\nonumber \\
&= \int_{\R^d\times \Gamma_t}g(\gamma(0))d\eta(x,\gamma)\nonumber \\
& \quad + \int^t_0\int_{\R^d\times\Gamma_t}\left(\frac{1}{p}|\dot\gamma(s)|^p -V(\gamma(s))\right)d\eta(x,\gamma)ds\nonumber \\&= \int_{\R^d\times\Gamma_t}\left\{g(\gamma(0))+\int^t_0\left(\frac{1}{p}|\dot{\gamma}(s)|^p -V(\gamma(s))\right)ds\right\}d\eta(x,\gamma)\nonumber \\
&\ge \int_{\R^d\times\Gamma_t}u(x,t)d\eta(x,\gamma)\\
&= \int_{\R^d}u(x,t)d\mu(x).\nonumber
\end{align}
The interchange of order of integration follows from the assumption \eqref{classVassump} and a routine application of Fubini's theorem. Thus, $\cU(\mu,t)\ge \int_{\R^d}u(x,t)d\mu(x).$

\par We now pursue the opposite inequality. We appeal to the lemma above to obtain the
family of Borel measurable mappings $\{\Psi(s)\}_{s\in[0, t]}$
\eqref{FlowMap} and also define the family of Borel probability
measures
$$
\sigma(s):=\Psi(s)_{\#}\mu, \quad s\in [0,t].
$$
Since $s\mapsto \Psi(x,s)$ is a minimizer for $u(x,t)$, the assumption \eqref{classVassump} implies
$$
\int_{\R^d}\int^t_0|\partial_s\Psi(x,s)|^pds d\mu(x)<\infty
$$
for each $\mu\in \cMp$; a proof of this bound follows closely with the proof of Corollary \ref{UFinite}.  For $0\le s_1< s_2<t$,
\begin{align*}
W_p(\sigma(s_1),\sigma(s_2))^p & \le \int_{\R^d}|\Psi(s_1,x)-\Psi(s_2,x)|^pd\mu(x) \\
& \le (s_2-s_1)^{p-1}\int_{\R^d}\int^{s_2}_{s_1}|\partial_s\Psi(\tau,x)|^pd\tau d\mu(x)
\end{align*}
which leads to
$$
\left(\frac{W_p(\sigma(s_2),\sigma(s_1))}{s_2-s_1}\right)^p\le \frac{1}{s_2-s_1}\int^{s_2}_{s_1} \left(\int_{\R^d}|\partial_s\Psi(\tau,x)|^pd\mu(x)\right)d\tau.
$$
It now follows that
$$
||\dot\sigma(s)||^p\le \int_{\R^d}|\partial_s\Psi(x,s)|^pd\mu(x)
$$
for Lebesgue a.e. $s\in [0,t]$.  As a result
\begin{align*}
\cU(\mu,t)&\le {\cal G}(\sigma(0))+\int^t_0\left(\frac{1}{p}||\dot{\sigma}(s)||^p -{\cal V}(\sigma(s))\right)ds \\
&= \int_{\R^d}g(\Psi(x,0))d\mu(x)+ \int^t_0\left(\frac{1}{p}||\dot{\sigma}(s)||^p -\int_{\R^d}V(\Psi(x,s))d\mu(x)\right)ds\\
&\le \int_{\R^d}\left\{g(\Psi(x,0)) +\int^t_0\left(\frac{1}{p}|\partial_s\Psi(x,s)|^p -V(\Psi(x,s))\right)ds\right\}d\mu(x)\\
&=\int_{\R^d}u(x,t)d\mu(x).
\end{align*}
$(ii)$ From the above calculations, it is clear that for each $(\mu,t)\in \cMp\times (0,T)$, the path $s\mapsto \Psi(s)_{\#}\mu$ is optimal for $\cU(\mu,t)$.
\end{proof}

\begin{ex}\label{Power2Ex}
For $p=2$, $g\equiv 0$ and $V(x)=\frac{1}{2}|x|^2$, the classical action is given  by $u(x,t)=-\tan(t)|x|^2/2$ for
$t\in (0,\pi/2)$, and so
$$
\cU(\mu,t)=-\tan(t)\int_{\R^d}\frac{|x|^2}{2}d\mu(x).
$$
In this case, the flow map is also explicit
$$
\Psi(x,s)=\frac{\cos(s)}{\cos(t)}x, \quad (x,s)\in \R^d\times [0,t].
$$
\end{ex}

\begin{ex}\label{PowerpEx}
Continuing from our example above, we now assume $p\in (1,\infty)$, $g\equiv 0$ and $V(x)=\frac{1}{p}|x|^p$. A good exercise is to
show that the classical value function is given  by
$$
u(x,t)=a(t)\frac{|x|^p}{p}, \quad \R^d\times(0,T_p).
$$
Here $a\in C^\infty(0,T_p)\cap C[0,T_p)$ is a negative solution of the initial value problem
$$
\begin{cases}
\dot{a}(t)+(p-1)|a(t)|^q+1=0, \quad 0<t<T_p\\
a(0)=0
\end{cases}
$$
with
$$
T_p:=\frac{\pi/q}{(p-1)^{1/q}\sin(\pi/q)}.
$$
We remark that it also follows from the construction of this solution that $\lim_{t\uparrow T_p}a(t)=-\infty$ and the flow map is given by
$$
\Psi(x,s)=\exp\left(\int^t_s |a(\tau)|^{\frac{1}{p-1}}d\tau\right) x, \quad (x,s)\in \R^d\times [0,t].
$$
Theorem \ref{thm2} then states the associated generalized value function $\cU(\mu,t)$ is obtained from integrating $x\mapsto u(x,t)$ against $\mu$.
\end{ex}

\section{Minimizing trajectories}\label{MinPaths}
We now consider the question of whether or not minimizing paths exist for generalized value functions.  In this section,
we will show that if $\cG$ and $\cV$ are continuous with respect to the narrow topology, then direct methods can be employed to verify the existence of minimizing paths.
We shall also provide a statement asserting precisely how the gradient of solutions of \eqref{classicalHJE} solve the Euler-Poisson equations when $\cG$ and $\cV$ satisfy
\eqref{narrowGV}.

\par We begin our study with a compactness result.  Recall that narrow convergence of Borel probability measures on $\R^d$ is completely metrizable (see Chapter 6 of \cite{Bill}).  As a particular metric on this space, we take the L\'{e}vy-Prokhorov metric
$$
d(\mu,\nu):=\inf\left\{\epsilon>0: \mu(A)\le \nu(A^\epsilon)+\epsilon, \; \nu(A)\le \mu(A^\epsilon)+\epsilon, \; \text{for all Borel}\; A\subset\R^d\right\}.
$$
Here $X_\epsilon:=\cup_{z\in X}B_{\epsilon}(z)$. The reason for this choice is due to the following inequality
\begin{equation}\label{deltaIneq}
d^2\le W_1
\end{equation}
(Corollary 2.18 of \cite{Huber}).

\begin{prop}
Let $p\in (1,\infty)$, $\mu\in \cMp$, and a sequence $\{\sigma^k\}_{k\in \N}\subset AC_p([0,t], \cMp)$, such that
$$
\sigma^k(t)=\mu, \quad k\in \N
$$
and
\begin{equation}\label{UnifEq}
\sup_{k\in \N}\int^t_0||\dot\sigma^k(s)||ds<\infty.
\end{equation}
Then there is a subsequence $\sigma^{k_j}$ and $\eta\in AC_p([0,t], \cMp)$ such that $\eta=\lim_{j\rightarrow\infty}\sigma^{k_j}$ in
$$
C([0,t],(\cPp,d)).
$$
Moreover
$$
\int^t_0||\dot\eta(s)||^pds\le \liminf_{j\rightarrow \infty}\int^t_0||\dot\sigma^{k_j}(s)||^pds.
$$
\end{prop}
\begin{rem}
Similar results have been established. For instance, see Proposition 3.5 and Theorem 5.2 in a recent preprint of Gangbo and \'{S}wi\k{e}ch \cite{GS} and Proposition 4 in \cite{G2}.  The novelty of the
above proposition is that it asserts the {\it uniform} convergence of an appropriate subsequence in the narrow topology.
\end{rem}
\begin{proof}
We employ the Arzel\`{a}-Ascoli diagonalization argument with some modifications. Let $\{s_n\}_{n\in \N}\subset [0,t]$ be dense. Observe that
$$
W_p(\sigma^k(s_1),\mu)=W_p(\sigma^k(s_1),\sigma(t))\le \int^t_{s_1}||\sigma^k(s)||ds\le C.
$$
As $|x|\le |x-y| +|y|$,
$$
\int_{\R^d}|x|d\sigma^k_{s_1}(x)\le W_p(\sigma^k(s_1),\mu)+\int_{\R^d}|y|d\mu(y)\le C+\int_{\R^d}|y|d\mu(y)
$$
$\{\sigma^k(s_1)\}_{k\in \N}$ is tight, and by Prokhorov's theorem, has a narrowly convergent subsequence which we will denote
$$
\{\sigma^k_1(s_1)\}_{k\in \N}.
$$
Continuing inductively, we obtain narrowly convergent sequences $\{\sigma^k_n(s_n)\}_{k\in \N}$  where
$$
\{\sigma^k_{n+1}\}_{k\in \N}\subset \{\sigma^k_n\}_{k\in \N}\subset \{\sigma^k\}_{k\in \N}
$$
for each $n\in \N$. Defining $\eta^k:=\sigma^k_k$ for $k\in \N$, we have that by construction that $\eta^k(s_n)$ is narrowly convergent for all $n\in \N$.

\par Employing the inequality \eqref{deltaIneq} and the bounds \eqref{UnifEq}, one checks that $\{\eta_k\}_{k\in \N}\subset C([0,t],(\cPp, d))$ is equicontinuous.  Now, fix
$\epsilon>0$ and choose $\rho(\epsilon)$ so that for all $k\in \N$
$$
d(\eta^k(\tau_1),\eta^k(\tau_2))<\epsilon
$$
for $|\tau_1-\tau_2|<\rho(\epsilon)$, $\tau_1, \tau_2\in [0,t]$. Also select $\{t_1, t_2, \dots, t_N\}\subset\{s_n\}$ such that any $s\in [0,t]$, there is $j\in \{1,2,\dots, N\}$ such that $|s-t_j|<\rho(\epsilon)$. Note that
for any $s\in [0,t]$,
\begin{align*}
d(\eta^k(s),\eta^\ell(s))\le d(\eta^k(s),\eta^k(t_j))+d(\eta^k(t_j),\eta^\ell(t_j))+d(\eta^\ell(t_j),\eta^\ell(s))<3\epsilon
\end{align*}
by choosing $k$ large enough. As the metric space $(\cPp, d)$ is complete, so is $C([0,t],(\cPp, d))$.  It follows that $\{\eta^k\}_{k\in \N}$ is uniformly convergent to some $\eta\in C([0,t],(\cPp, d))$.

\par By \eqref{UnifEq}, the sequence of functions $s\mapsto ||\dot\eta^k(s)||$ is bounded in $L^p(0,t)$ and so has a subsequence (which we will not relabel) that converges weakly to some $g\in L^p(0,t)$. By narrow convergence,
$$
W_p(\eta(s), \eta(s'))\le \liminf_{k\rightarrow\infty}W_p(\eta^k(s), \eta^k(s'))\le \liminf_{k\rightarrow \infty}\int^s_{s'}||\dot\eta^{k}(\tau)||d\tau=\int^s_{s'}g(\tau)d\tau
$$
for $0\le s\le s'\le t$ (Lemma 7.1.4 of \cite{AGS}).  Hence, $\eta \in AC_p([0,t], \cMp)$ and $||\dot\eta||\le g$. By weak convergence in $L^p(0,t)$,
$$
\int^t_0||\dot\eta(s)||^pds\le \int^t_0 g(s)^pds\le \liminf_{k\rightarrow\infty}\int^t_0||\dot\eta^k(s)||^pds.
$$
\end{proof}

\begin{cor}\label{Existence}
Suppose ${\cal G}$ and ${\cal V}$ are narrowly continuous and that $T$ satisfies \eqref{TChoice}.  Then $\cU(\mu,t)$ has a minimizing path for each $(\mu,t)\in \cMp\times [0,T)$.
\end{cor}
\begin{proof}
Let $\sigma_k\in AC_p([0,t], \cMp)$ be a minimizing sequence for $\cU(\mu,t)$. Without loss of generality, suppose there is a sequence of positive numbers $\epsilon_k$ tending to $0$ such that
\begin{equation}\label{EpsK}
\cU(\mu,t)>-\epsilon_k + {\cal G}(\sigma^k(0))+\int^t_0\left(\frac{1}{p}||\dot{\sigma}^k(s)||^p -{\cal V}(\sigma^k(s))\right)ds.
\end{equation}
The simple bound $\cU(\mu,t)\le \cG(\mu)-t\cV(\mu)$ and the above inequality \eqref{EpsK}, manipulated as in \eqref{PoincareLower}, imply there is a universal constant $C$ such that
$$
\int^t_0||\dot{\sigma}^k(s)||^p\le C, \quad k\in \N.
$$
As $\sigma^k(t)=\mu$, the sequence satisfies the hypotheses of the previous lemma and so this sequence converges (up to a subsequence) uniformly in the narrow topology to a $\sigma\in AC_p([0,t],\cMp)$. Moreover,
$$
\int^t_0||\dot\sigma(s)||^pds\le \lim_{k\rightarrow \infty}\int^t_0||\dot\sigma^{k}(s)||^pds.
$$
With our narrow continuity assumptions on $\cG$ and $\cV$, we are now able to pass to the limit as $k\rightarrow \infty$ in \eqref{EpsK} and conclude
$$
\cU(\mu,t)\ge  {\cal G}(\sigma(0))+\int^t_0\left(\frac{1}{p}||\dot{\sigma}(s)||^p -{\cal V}(\sigma(s))\right)ds.
$$
\end{proof}
Let us now assume that $\cG$ and $\cV$ satisfy \eqref{narrowGV} with $g,V\in C^1(\R^d)$. As noted, any minimizing path for $u(x,t)$ satisfies the Euler-Lagrange equations
\eqref{ClassEulLag}. More generally, one derives
$$
\begin{cases}
\frac{d}{ds}\left(|\dot\gamma(s)|^{p-2}\dot\gamma(s)\right)=-\nabla V(\gamma(s)), \quad 0<s<t\\
\gamma(t)=x\\
\dot\gamma(0)=\nabla g(\gamma(0))
\end{cases}.
$$
For minimizing paths $\sigma$ of the generalized value $\cU(\mu,t)$ (with potential $\cV$ and initial condition $\cG$), the analogous necessary condition is that $\sigma$ and any corresponding velocity satisfy the Euler-Poisson system with appropriate boundary conditions. This is described below in the following proposition. We will not provide a proof as a very similar argument is provided in the proof of Theorem 3.9 of \cite{G}.

\begin{prop}\label{EPprop}
Assume \eqref{narrowGV}, $(\mu,t)\in \cMp\times [0,\infty)$ and that $\sigma$ is a minimizing path for $\cU(\mu,t)$. For the minimal velocity $v$ for $\sigma$,
$$
\begin{cases}
\hspace{1.27in} \partial_s\sigma +\nabla\cdot(\sigma v)=0\\
\partial_s(\sigma |v|^{p-2}v)+\nabla\cdot (\sigma |v|^{p-2}v\otimes v)=-\sigma \nabla V
\end{cases},\quad (x,s)\in \R^d\times (0,t)
$$
in the sense of distributions and
$$
|v(x,0)|^{p-2}v(x,0)=\nabla g(x), \quad \sigma(0)\; \text{a.e. $x$}\in \R^d.
$$
\end{prop}
We will, however, give a proof of Proposition \ref{vequalDu} which establishes a clear link between the classical
HJE \eqref{classicalHJE} and the Euler-Poisson system described in Proposition \ref{EPprop} above.

\begin{proof} (of Propsosition \eqref{vequalDu}) Assume $\sigma\in AC_p([0,t]; \cMp)$ is an optimal path for $\cU(\mu,t)$.  Let $\eta$ be a Borel probability measure on $\R^d\times \Gamma_t$ such that
$\sigma(s)=e(s)_{\#}\eta$ and $\eta$ is concentrated on pairs $(x,\gamma)$ for which $\gamma$ is a solution of the ODE \eqref{gammadotv} satisfying $\gamma(t)=x$.

\par Since $\sigma$ is optimal for $\cU(\mu,t)$, we may repeat the argument given in the proof of Theorem \eqref{thm2} and obtain an equality in \eqref{ImportantVequalDu}. It follows that for $\eta$ almost every
$(x,\gamma)\in \R^d\times\Gamma_t$, $\gamma$ is optimal for $u(x,t)$
$$
u(x,t)=g(\gamma(0)) +\int^t_0\left(\frac{1}{p}|\dot{\gamma}(s)|^p -V(\gamma(s))\right)ds.
$$
By assumption, $u$ is differentiable at $(\gamma(s),s)$ and so $\gamma$ satisfies the optimality equation \eqref{ClassOptEq}. In particular
$$
|v(\gamma(s),s)|^{p-2}v(\gamma(s),s)=\nabla u(\gamma(s),s), \quad \text{a.e. } s\in (0,t).
$$
Thus
$$
|v(e_s(x,\gamma)),s)|^{p-2}v(e_s(x,\gamma),s)=\nabla u(e_s(x,\gamma),s), \quad \text{a.e. } s\in (0,t)
$$
for $\eta$ almost every $(x,\gamma)\in \R^d\times\Gamma_t$, which completes the proof.
\end{proof}

\section{Hamilton-Jacobi equations}\label{SupersolnProof}
Our objective in this section is to prove Theorem \ref{thm1}. Therefore, it is appropriate that we start with a definition of viscosity solution of HJE in the Wasserstein spaces.
Our proof and definition will involve the {\it tangent space}
$$
\text{Tan}_\mu\cMp:=\overline{\{|\nabla \psi|^{q-2}\nabla\psi: \psi\in C^\infty_c(\R^d)\}}^{L^p(\mu)}
$$
and the {\it cotangent space}
$$
\text{CoTan}_\mu\cMp:=\overline{\{ \nabla \psi: \psi\in C^\infty_c(\R^d)\}}^{L^q(\mu)}
$$
of $\cMp$ at a measure $\mu$.

\par These spaces are natural to consider as an absolutely continuous path $\sigma\in AC_p([0,t],\cMp)$ always possesses a velocity field $v$
satisfying $v(s)\in \text{Tan}_{\sigma(s)}\cMp$ for Lebesgue almost every $s\in [0,t]$. And, of course, elements of $\text{Tan}$ and $\text{CoTan}$ can
be paired in a natural way. Nevertheless,  we emphasize that \text{Tan} and \text{CoTan} are definitions and that we are not asserting the existence of any
type of differentiable structure on $\cMp$.  The interested reader can consult section 8.4 of \cite{AGS} for more on tangent and cotangent spaces of the Wasserstein spaces.

\begin{defn}\label{ViscDefn}
$(i)$ $\cU\in USC(\cMp\times (0,T))$ is a {\it viscosity subsolution} of \eqref{WassHJE} if for each $(\mu_0,t_0)\in \cMp\times(0,T)$, $a\in\R$ and $\xi\in \text{CoTan}_\mu\cMp$ such
that
\begin{equation}\label{SubSolnIneq}
\cU(\mu,t)\le \cU(\mu_0,t_0) + \inf_{\pi\in\Gamma_0(\mu_0,\mu)}\iint_{\R^d\times\R^d}\xi(x)\cdot(y-x)d\pi(x,y) + a(t-t_0)+o(|t-t_0|) + o(W_p(\mu,\mu_0)),
\end{equation}
the following inequality holds
\begin{equation}\label{SubSolnIneq2}
a+\frac{1}{q}||\xi||^q_{L^q(\mu)}+{\cal V}(\mu_0)\le 0.
\end{equation}
Here $\Gamma_0(\mu_0,\mu)\subset \Gamma(\mu_0,\mu)$ is the collection of optimal measures for $W_p(\mu_0,\mu)$. \\\\
$(ii)$ $\cU\in LSC(\cMp\times (0,T))$ is a {\it viscosity supersolution} of \eqref{WassHJE} if for each $(\mu_0,t_0)\in \cMp\times(0,T)$, $a\in\R$ and $\xi\in \text{CoTan}_\mu\cMp$ such
that
\begin{equation}\label{SupSolnIneq}
\cU(\mu,t)\ge \cU(\mu_0,t_0) + \sup_{\pi\in\Gamma_0(\mu_0,\mu)}\iint_{\R^d\times\R^d}\xi(x)\cdot(y-x)d\pi(x,y) + a(t-t_0)+o(|t-t_0|) + o(W_p(\mu,\mu_0)),
\end{equation}
the following inequality holds
\begin{equation}\label{SupSolnIneq2}
a+\frac{1}{q}||\xi||^q_{L^q(\mu)}+{\cal V}(\mu_0)\ge 0.
\end{equation}
\\
$(iii)$  $\cU\in C(\cMp\times (0,T))$ is a {\it viscosity solution} of \eqref{WassHJE} if it is both a sub- and supersolution.
\end{defn}


\begin{proof} (of Theorem \ref{thm1}) 1. We first verify that $\cU$ is a subsolution of equation \eqref{WassHJE}. Assume \eqref{SubSolnIneq} holds for some $(\mu_0,t_0)\in \cMp\times(0,T)$, $a\in\R$ and $\xi\in \text{CoTan}_\mu\cMp$.  Set
$$
v:=\lambda\left(r-\text{id}_{\R^d}\right),
$$
for $\lambda>0$ and a Borel mapping $r:\R^d\rightarrow \R^d$ such that $(\id\times r)_{\#}\mu_0\in\Gamma_0(\mu_0,r_{\#}\mu_0)$. Theorem 8.5.1 of \cite{AGS} asserts that
the collection of such $v$ is $L^p(\mu_0)$ dense in $\text{Tan}_{\mu_0}\cMp$.

\par Next, define
$$
\sigma(s):=(\id + (t_0-s)v)_{\#}\mu_0, \quad s\in \R,
$$
and observe
$$
\id + (t_0-s)v=(1-\lambda(t_0 -s))\id + \lambda(t_0-s)r
$$
is a convex combination provided $t_0-1/\lambda<s< t_0$. It follows that the restriction of $\sigma$
to the interval $[t_0-1/\lambda,t_0]$ is a constant speed geodesic joining $\sigma(t_0-1/\lambda)=r_{\#}\mu_0$ to $\sigma(t_0)=\mu_0$ (Theorem 7.2.2 \cite{AGS}).
In particular,
$$
\pi(s):=\left(\id\times(\id + (t_0-s)v)\right)_{\#}\mu_0\in \Gamma_0(\mu_0, \sigma(s))
$$
for $s\in [t_0-1/\lambda,t_0]$.

\par Fix $h\in (0,\max\{t_0,1/\lambda\})$ and note that dynamic programming \eqref{DPP} combined with \eqref{SubSolnIneq} implies
\begin{align*}
\cU(\mu_0, t_0) & \le \cU(\sigma(t_0-h),t_0)+\int^{t_0}_{t_0-h}\left(\frac{1}{p}||\dot{\sigma}(s)||^p -{\cal V}(\sigma(s))\right)ds\\
& = \cU((\id + hv)_{\#}\mu_0, t_0-h) + h\frac{||v||^p_{L^p(\mu_0)}}{p} -\int^{t_0}_{t_0-h}{\cal V}(\sigma(s))ds\\
 &\le U(\mu_0, t_0) +\iint_{\R^d\times\R^d}\xi(x)\cdot(y-x)d\pi_{t_0-h}(x,y) - a h + h\frac{||v||^p_{L^p(\mu_0)}}{p} -\int^{t_0}_{t_0-h}{\cal V}(\sigma(s))ds +o(h) \\
& = U(\mu_0, t_0) +h\int_{\R^d}\xi(x)\cdot v(x)d\mu_0(x) - a h + h\frac{||v||^p_{L^p(\mu_0)}}{p} -\int^{t_0}_{t_0-h}{\cal V}(\sigma(s))ds +o(h)
\end{align*}
as $h\rightarrow 0^+$. Consequently, if we cancel $\cU(\mu_0, t_0)$, divide by $h$, and then send $h$ to 0,
\begin{equation}\label{MustTakeSup}
a- \int_{\R^d}\xi(x)\cdot v(x)d\mu_0(x)-\frac{1}{p}\int_{\R^d} |v(x)|^pd\mu_0(x)+{\cal V}(\mu_0)\le 0.
\end{equation}
As previously remarked, Theorem 8.5.1 of \cite{AGS} implies the above inequality holds for a $L^p(\mu_0)$ dense set of $v\in \text{Tan}_{\mu_0}\cMp$;
consequently, it must also hold for all $v\in \text{Tan}_{\mu_0}\cMp$.  Taking the supremum all $v\in\text{Tan}_{\mu_0}\cMp$ in \eqref{MustTakeSup} verifies \eqref{SubSolnIneq2} as
desired.
\par 2. We now show that $\cU$ is a supersolution of equation \eqref{WassHJE}. Assume \eqref{SupSolnIneq} holds for some $(\mu_0,t_0)\in \cMp\times(0,T)$, $a\in\R$ and $\xi\in \text{CoTan}_\mu\cMp$, and also fix $\delta\in (0,1)$. For each $h\in (0,t_0)$, there is $\sigma=\sigma^h$ admissible for $\cU(\mu_0,t_0)$ such that
\begin{equation}\label{LowerSup}
\cU(\mu_0, t_0) > -\delta h + {\cal G}(\sigma(0)) +\int^{t_0}_{0}\left(\frac{1}{p}||\dot{\sigma}(s)||^p -{\cal V}(\sigma(s))\right)ds.
\end{equation}
Moreover,
\begin{equation}\label{hdeltaineq}
\cU(\mu_0, t_0) > -\delta h + \cU(\sigma(t_0-h),t_0-h) +\int^{t_0}_{t_0-h}\left(\frac{1}{p}||\dot{\sigma}(s)||^p -{\cal V}(\sigma(s))\right)ds
\end{equation}
also holds for $h\in (0,t_0)$.

\par Recall the inequality $\cU(\mu_0,t_0)\le {\cal G}(\mu_0) -t_0 {\cal V}(\mu_0)$, which follows from choosing the constant path equal to $\mu_0$ in the definition of $\cU$. Combining this upper bound with \eqref{LowerSup} and \eqref{PoincareLower}, implies
$$
\int^{t_0}_0||\dot\sigma(s)||^pds\le C
$$
independently $h\in(0,t_0)$. In particular, we have the estimate
\begin{equation}\label{BadEstW}
W_p(\sigma(t_0-h), \mu_0)\le \int^{t_0}_{t_0 -h}||\dot{\sigma}(s)||ds\le Ch^{1-1/p}
\end{equation}
for some universal constant $C$.

\par By inequality \eqref{hdeltaineq} and our assumption \eqref{SupSolnIneq},
\begin{align}\label{supsolnineq1}
\cU(\mu_0, t_0) & > -\delta h + \cU(\mu_0, t_0) -a h +  \iint_{\R^d\times\R^d}\xi(x)\cdot(y-x)d\pi_h(x,y) \nonumber \\
&+\int^{t_0}_{t_0-h}\left(\frac{1}{p}||\dot{\sigma}(s)||^p -{\cal V}(\sigma(s))\right)ds +o(h)+o(W_p(\sigma(t_0-h), \mu_0))
\end{align}
for any $\pi_h\in \Gamma_0(\mu_0,\sigma(t_0-h))$. In particular, we may cancel the $\cU(\mu_0, t_0)$ terms, divide by $h$ and employ 
\eqref{BadEstW} to arrive at 
\begin{align}\label{supsolnineq}
0 & > -\delta -a  + \frac{1}{h} \iint_{\R^d\times\R^d}\xi(x)\cdot(y-x)d\pi_h(x,y) \nonumber \\
&+\frac{1}{h}\int^{t_0}_{t_0-h}\left(\frac{1}{p}||\dot{\sigma}(s)||^p -{\cal V}(\sigma(s))\right)ds +o(1)+o(1)\frac{W_p(\sigma(t_0-h), \mu_0)}{h}
\end{align}
as $h\rightarrow 0^+$. 

\par  Applying the Poincar\'{e} inequality \eqref{Poincare}, we find for any $\epsilon>0$
\begin{align*}
\int^{t_0}_{t_0-h}{\cal V}(\sigma(s))ds &\le \int^{t_0}_{t_0-h}\left\{\alpha W_p(\sigma(s),\varrho)^p+\beta\right\}ds \\
&\le \int^{t_0}_{t_0-h}\alpha 2^pW_p(\sigma(s),\mu_0)^pds +\left(\beta+\alpha 2^pW_p(\mu_0,\varrho)^p\right)h \\
&\le \alpha 2^p(C_p h)^p\int^{t_0}_{t_0-h}|| \dot\sigma(s)||^pds +\left(\beta+\alpha 2^pW_p(\mu_0,\varrho)^p\right)h \\
&\le \epsilon \int^{t_0}_{t_0-h}|| \dot\sigma(s)||^pds +  Ch
\end{align*}
for $h$ small enough.

\par Young's inequality also gives
$$
\iint_{\R^d\times\R^d}\xi(x)\cdot(y-x)d\pi_h(x,y)\ge -\frac{2h}{q}||\xi||^q_{L^q(\mu_0)} -\frac{W_p(\sigma(t_0-h),\mu_0)^p}{p2^{p-1}h^{p-1}}
$$
Combining these observations with \eqref{supsolnineq} yields
\begin{align*}
0&> -\delta-a -\frac{2}{q}||\xi||^q_{L^q(\mu_0)}  +o(1)\left(1+ \frac{W_p(\sigma(t_0-h),\mu_0)}{h}\right)\\
&  +\left(\frac{1}{p}-\epsilon-\frac{1}{2^{p-1}p}\right)\left(\frac{W_p(\sigma(t_0-h),\mu_0)}{h}\right)^p-  C
\end{align*}
It is now immediate that for $\epsilon>0$ chosen small enough (which can be achieved by choosing $h$ small enough),
\begin{equation}\label{UnifSmallh}
W_p(\sigma(t_0-h),\mu_0)\le C h
\end{equation}
for all sufficiently small $h$.  And in particular, $o(W_p(\sigma(t_0-h,\mu_0))=o(h)$; note the improvement over \eqref{BadEstW} and the control we now have over the term 
$o(1)\frac{W_p(\sigma(t_0-h), \mu_0)}{h}$ in \eqref{supsolnineq}.

\par 3. With our new estimate, inequality \eqref{supsolnineq1} becomes
\begin{align*}
\cU(\mu_0, t_0) &> -\delta h + \cU(\mu_0, t_0) -a h +  \iint\xi(x)\cdot(y-x)d\gamma_h(x,y) \nonumber \\
&+\int^{t_0}_{t_0-h}\left(\frac{1}{p}||\dot{\sigma}(s)||^p -{\cal V}(\sigma(s))\right)ds +o(h)\\
&> -\delta h + \cU(\mu_0, t_0) -a h - \frac{h}{q}||\xi||^q_{L^q(\mu_0)}-\frac{W_p(\sigma(t_0-h),\mu_0)^p}{ph^{p-1}} \nonumber \\
&+\int^{t_0}_{t_0-h}\left(\frac{1}{p}||\dot{\sigma}(s)||^p -{\cal V}(\sigma(s))\right)ds +o(h)\\
&\ge -\delta h + \cU(\mu_0, t_0) -a h - \frac{h}{q}||\xi||^q_{L^q(\mu_0)}-\int^{t_0}_{t_0-h}{\cal V}(\sigma(s))ds +o(h).
\end{align*}Hence,
$$
a+\frac{1}{q}||\xi||^q_{L^q(\mu_0)}+\frac{1}{h}\int^{t_0}_{t_0-h}{\cal V}(\sigma(s))ds> -\delta +o(1)
$$
as $\delta \rightarrow 0^+$. Notice that $\lim_{s\rightarrow t_0^-}\sigma(s)=\mu_0$ {\it uniformly} in all $h>0$ small enough by \eqref{UnifSmallh}.
Thus,
$$
a+\frac{1}{q}||\xi||^q_{L^q(\mu_0)}+{\cal V}(\mu_0)\ge -\delta.
$$
As $\delta$ was arbitrary, inequality \eqref{SupSolnIneq2} is now established.
\par 4.  Assertion $(ii)$ follows from simple estimates. Indeed, recall $\cU(\mu,t)\le {\cal G}(\mu)-t{\cal V}(\mu)$ and that in the proof of Lemma \eqref{UFinite}, we established
an inequality that implies $\cU(\mu,t)\ge {\cal G}(\mu) - Ct$ for an appropriate constant $C=C(\mu)$. Hence,  $\lim_{t\rightarrow 0^+}\cU(\mu,t)=\cG(\mu)$.
\end{proof}
\begin{rem} Following Definition \ref{ViscDefn}, we may define the {\it superdifferential} of ${\cal H}\in USC({\cal M}_2)$ at $\mu_0$ as mappings $\xi\in\text{CoTan}_{\mu_0}({\cal M}_2)$
for which
\begin{equation}\label{subdiffuno}
{\cal H}(\mu)\le {\cal H}(\mu_0)+\inf_{\pi\in \Gamma_0(\mu_0,\mu)}\iint_{\R^d\times\R^d}\xi(x)\cdot(y-x)d\pi(x,y) + o(W_2(\mu_0,\mu)),
\end{equation}
as $W_2(\mu_0,\mu)\rightarrow 0$.  We may extend
this definition to measures as follows.  A measure $\gamma\in {\cal P}_2(\R^d\times\R^d)$ belongs to the {\it extended subdifferential} of ${\cal H}$  at $\mu_0$ provided $\pi^1_{\#}\gamma=\mu_0$ and
\begin{equation}\label{subdiffdos}
{\cal H}(\mu)\le {\cal H}(\mu_0)+\inf_{\rho}\iiint_{\R^d\times\R^d\times \R^d}z\cdot(y-x)d\rho(x,y,z) + o(W_2(\mu,\mu_0)),
\end{equation}
as $W_2(\mu,\mu_0)\rightarrow 0$. Above, the supremum is taken
over $\rho\in {\cal P}_2(\R^d\times\R^d\times \R^d)$ satisfying
$\pi^{1,3}_\#\rho=\gamma$ and
$\pi^{1,2}_\#\rho\in\Gamma_0(\mu_0,\mu)$.

\par Observe that for any $\xi$ satisfying \eqref{subdiffuno}, $\gamma:=(\id\times\xi)_\#\mu_0$ satisfies \eqref{subdiffdos}.  In \cite{AF}, Ambrosio and Feng used the extended subdifferential to successfully compare the corresponding viscosity sub and supersolutions. 
It is an open problem to deduce whether viscosity sub and supersolutions as defined in this paper admit a comparison principle. 
\end{rem}

\par We conclude by stating that the modified generalized value function \eqref{ModValue} also has a PDE characterization.  We omit the proof as it is very similar to the one presented above. Again, we view this a very good sign that there are many interesting problems to be worked out in the direction of deterministic control in the Wasserstein
spaces.

\begin{prop}\label{LaxProp} Assume $\ell$ is increasing, convex and satisfies
$$
\lim_{w\rightarrow +\infty}\frac{\ell(w)}{w}=+\infty.
$$
Then the modified action $\cU$ defined in \eqref{ModValue} is a viscosity solution of the HJE \eqref{LaxPDE}.
\end{prop}



\begin{thebibliography}{}

\bibitem{AB} Aubin, J.P; Frankowska, H. \emph{Set-Valued Analysis}. Birkhauser, Boston, 1990.

\bibitem{AGS} Ambrosio, Luigi; Gigli, Nicola; SavarŽ\'{e}, Giuseppe. \emph{Gradient flows in metric spaces and in the space of probability measures}. Second edition. Lectures in Mathematics ETH ZŸrich. BirkhŠuser Verlag, Basel, 2008.


\bibitem{AF} Ambrosio, L.; Feng, J. \emph{On a class of first order Hamilton-Jacobi equations in metric space}. Journal of differential equations, Vol. 256 (2014), no. 7, 2194--2245.

\bibitem{Bardi} Bardi, M.; Capuzzo-Dolcetta, I. \emph{Optimal control and viscosity solutions of Hamilton-Jacobi-Bellman equations}. With appendices by Maurizio Falcone and Pierpaolo Soravia. Systems \& Control: Foundations \& Applications. BirkhŠuser Boston, Inc., Boston, MA, 1997.

\bibitem{Bill}  Billingsley, Patrick. \emph{Convergence of probability measures}. Second edition. Wiley Series in Probability and Statistics: Probability and Statistics. A Wiley Interscience Publication. John Wiley $\&$ Sons, Inc., New York, 1999.

\bibitem{CIL} Crandall, M. G.;Ishii, H.;Lions, P.-L. \emph{User's guide to viscosity solutions of second order partial differential equations}. Bull. Amer. Math. Soc. (N.S.) 27 (1992), no. 1, 1--67.

\bibitem{Drag} Dragoni, Federica.  \emph{Metric Hopf-Lax formula with semicontinuous data}. Discrete Contin. Dyn. Syst. 17 (2007), no. 4, 713--729.

\bibitem{FK} Feng, J.; Katsoulakis, M. \emph{A comparison principle for HamiltonÐJacobi equations related to controlled gradient flows in infinite dimensions}. Arch. Ration. Mech. Anal., 192 (2009), no. 2, 275--310.

\bibitem{FKu} Feng, J.; Kurtz, T. \emph{Large deviations for stochastic processes}. Mathematical Surveys and Monographs, American Mathematical Society, Providence, RI, 2006.

\bibitem{FN} Feng, J.; Nguyen, T. \emph{Hamilton--Jacobi equations in space of measures associated with a system of conservation laws}. J. Math. Pures Appl., 97 (2012), no. 4, 318--390.

\bibitem{FS} Feng, J; \'{S}wi\k{e}ch, A. \emph{Optimal control for a mixed flow of Hamiltonian and gradient type in space of probability measures}. Trans. Amer. Math. Soc. 365 (2013), no. 8, 3987--4039.

\bibitem{FR} Fleming, Wendell H.; Rishel, Raymond W. \emph{Deterministic and stochastic optimal control}. Applications of Mathematics, No. 1. Springer-Verlag, Berlin-New York, 1975.

\bibitem{FS} Fleming, Wendell H.; Soner, H. Mete. \emph{Controlled Markov processes and viscosity solutions}. Second edition. Stochastic Modelling and Applied Probability, 25. Springer, New York, 2006.

\bibitem{G} Gangbo, W; Nguyen, T; Tudorascu, A. \emph{Hamilton-Jacobi equations in the Wasserstein space}.  Methods Appl. Anal. 15 (2008), no. 2, 155--183.

\bibitem{G2} Gangbo, W.; Nguyen, T.; Tudorascu, A. \emph{Euler-Poisson systems as action-minimizing paths in the Wasserstein space}. Arch. Ration. Mech. Anal. 192 (2009), no. 3, 419--452.

\bibitem{GS} Gangbo, W.; \'{S}wi\k{e}ch, A. \emph{Optimal transport and large number of particles}. Discrete Contin. Dyn. Syst. 34 (2014), no. 4, 1397--1441.


\bibitem{GS2} Gangbo, W.; \'{S}wi\k{e}ch, A. \emph{Metric viscosity solutions of Hamilton--Jacobi equations depending on local slopes}.  to appear in Calc. Var. Partial Differential Equations.


\bibitem{GHN} Giga,Y; Hamamuki, N; Nakayasu, A. \emph{Eikonal equations in metric spaces}. Trans. Amer. Math. Soc. 367 (2015), no. 1, 49--66.

\bibitem{Gos} Gozlan, N;  Roberto, C; Samson, P-M. \emph{Hamilton-Jacobi equations on metric spaces and transport-entropy inequalities}. Revista Matematica Iberoamericana - \'{a} paraître (2012).

\bibitem{Huber}  Huber, Peter J.; Ronchetti, Elvezio M. \emph{Robust statistics}. Second edition. Wiley Series in Probability and Statistics. John Wiley $\&$ Sons, Inc., Hoboken, NJ, 2009

\bibitem{LL} Lasry, J.-M., Lions, P.-L. \emph{Mean field games}. Japan. J. Math. 2 (2007), no. 1, 229--260.

\bibitem{LL2} Lasry, J.-M., Lions, P.-L. \emph{Jeux \'{a} champ moyen II. Horizon fini et contr\^{o}le optimal}. C. R. Math. Acad. Sci. Paris 343 (2006), no. 10, 679--684.

\bibitem{V} Villani, C. \emph{Optimal transport. Old and new}. Grundlehren der Mathematischen Wissenschaften [Fundamental Principles of Mathematical Sciences], 338. Springer-Verlag, Berlin, 2009.

\end{thebibliography}
\end{document}